\documentclass[10pt]{amsart}
\usepackage{amsmath}
\usepackage[usenames,dvipsnames]{color}
\usepackage{parskip}
\usepackage{amsfonts}
\usepackage{amscd}
\usepackage[centertags]{amsmath}
\usepackage{amssymb}
\usepackage[all,cmtip]{xy}
\usepackage[english]{babel}
\usepackage{tabularx}
\usepackage{mathtools}
\usepackage{amsxtra}
\usepackage{euscript}
\usepackage[T1]{fontenc}
\usepackage{doc, exscale, fontenc, latexsym, syntonly}
\usepackage{amsfonts}
\usepackage{amsthm}
\usepackage{graphicx}
\usepackage{xcolor}
\usepackage{tikz-cd}

\newtheorem{theorem}{Theorem}[section]
\newtheorem{corollary}[theorem]{Corollary}
\newtheorem{lemma}[theorem]{Lemma}
\newtheorem{proposition}[theorem]{Proposition}
\newtheorem{remark}[theorem]{Remark}
\theoremstyle{definition}
\newtheorem{definition}[theorem]{Definition}

\numberwithin{figure}{section}
\numberwithin{table}{section}

\newcommand{\cat}{\ensuremath{\mathrm{cat}}}
\newcommand{\secat}{\ensuremath{\mathrm{secat}}}
\newcommand{\id}{\ensuremath{\mathrm{Id}}}
\newcommand{\TC}{\ensuremath{\mathrm{TC}}}
\newcommand{\D}{\ensuremath{\mathrm{D}}}
\newcommand{\pr}{\ensuremath{\mathrm{pr}}}

\begin{document}

\title{Higher Homotopic Distance}

\author{Ayse Borat and Tane Vergili}

\date{\today}

\address{\textsc{Ayse Borat}
Bursa Technical University\\
Faculty of Engineering and Natural Sciences\\
Department of Mathematics\\
Bursa, Turkey}
\email{ayse.borat@btu.edu.tr} 

\address{\textsc{Tane Vergili}
Ege University\\
Faculty of Science\\
Department of Mathematics\\
Izmir, Turkey}
\email{tane.vergili@ege.edu.tr}

\subjclass[2010]{}


\begin{abstract} The concept of homotopic distance and its higher analog are introduced in \cite{MVML}. In this paper we introduce some important properties of higher homotopic distance, investigate the conditions under which $\cat$, $\secat$ and higher dimensional topological complexity are equal to the higher homotopic distance, and give alternative proofs, using higher homotopic distance, to some $\TC_n$-related theorems.
\end{abstract}

\maketitle

\section{Introduction}

The notion of homotopic distance is first introduced by Macias-Virgos and Mosquera-Lois in \cite{MVML}, which is a homotopy invariant whose special cases are topological complexity ($\TC$) and Lusternik-Schnirelmann category ($\cat$), and is defined as follows. 

\begin{definition}\cite{MVML} Given two maps $f, g: X\rightarrow Y$, the homotopic distance between $f$ and $g$ is the least non-negative integer $k$ such that one can find an open cover $\{U_0, \cdots, U_k\}$ for $X$ satisfying $f\big|_{U_i}\simeq g\big|_{U_i}$ for all $i=0,1, \cdots, k$. It is denoted by $\D(f,g)$. If there is no such a covering, we write $\D(f,g)=\infty$. 
\end{definition}

The organization of the paper is as follows: 

In Section 2, we will recall the higher homotopic distance and introduce some propositions and lemmas. These lemmas will be mainly used in Section 3 and Section 4 to prove main theorems of this paper. 

Motivated from the fact that homotopic distance has a relation between topological complexity and Lusternik-Schnirelmann category, we will show that $n$-th homotopic distance of some specific maps is equal to $n$-th topological complexity and Lusternik-Schnirelmann category. Moreover, we will give the relation between $n$-th homotopic distance and sectional category. Later in the same section, we will give alternative proofs of the well-known theorems about $\TC_n$. 

In the last section, we will prove the homotopy invariance of higher homotopic distance and deduce that $\TC_n$ is homotopy invariant. 

For a further reading about the variances of homotopic distance, we refer the interested readers to see \cite{MVML2} in which categorical version of homotopic distance between functors is introduced. 

\section{Higher Homotopic Distance and Some of Its Properties}

Higher homotopic distance, as well as usual homotopic distance, was first introduced in \cite{MVML} by Macias-Virgos and Mosquera-Lois. In this section, we will recall its definition and introduce some of its properties which are of importance to prove some main theorems.

\begin{definition}\cite{MVML} Given $f_i: X\rightarrow Y$ for $i\in \{1, 2, \cdots, n\}$, the $n$-th homotopic distance $\D(f_1, f_2, \cdots, f_n)$ is the least non-negative integer $k$ such that there exist open subsets $U_0, U_1, \cdots, U_k$ which covers $X$ and satisfy $f^j_1|_{U_j}\simeq f^j_2|_{U_j}\simeq \cdots \simeq f^j_n|_{U_j}$ for all $j\in\{0, 1, \cdots, k\}$.

If there is no such a covering, we define $\D(f_1, f_2, \cdots, f_n)=\infty$. 
\end{definition}

The following four propositions are direct consequences of the definition. 

\begin{proposition} $\D(f_1, f_2, \cdots, f_n)= \D(f_{\sigma(1)}, f_{\sigma(2)}, \cdots, f_{\sigma(n)})$ holds for any permutation $\sigma$ of $\{1, 2, \cdots, n\}$.
\end{proposition}
\qed

\begin{proposition} $\D(f_1, f_2, \cdots, f_n)=0$ iff $f_i\simeq f_{i+1}$ for each $i \in \{1, 2, \cdots, n-1\}$.
\end{proposition}
\qed

\begin{proposition}\label{A} Given maps $f_i: X\rightarrow Y$ and $g_i: X\rightarrow Y$ for $i\in \{1, 2, \cdots, n\}$. If $f_i\simeq g_i$ for each $i$, then $\D(f_1, f_2, \cdots, f_n)= \D(g_1, g_2, \cdots, g_n)$. 
\end{proposition}
\qed

\begin{proposition} If  $1 < m < n$ and $f_1, f_2, \cdots, f_m, \cdots, f_n: X \rightarrow Y$ are maps, then $\D(f_1, f_2, \cdots, f_m)\leq \D(f_1, f_2, \cdots, f_n)$.
\end{proposition}
\qed

\begin{proposition} If $f_1, f_2, \cdots, f_n: X \rightarrow Y$ are maps and if $\{U_0, U_1, \cdots, U_k\}$ is any open covering of $X$, then we have
\[
\D(f_1, f_2, \cdots, f_n) \leq \sum_{i=0}^k \D(f_1\big|_{U_i}, f_2\big|_{U_i}, \cdots, f_n\big|_{U_i}) + k.
\]
\end{proposition}

\begin{proof} Let $\D(f_1\big|_{U_i}, f_2\big|_{U_i}, \cdots, f_n\big|_{U_i})=m_i$ for all $i\in \{0, 1, \cdots, k\}$. Then there exists an open covering $\{U^0_i, U^1_i, \cdots, U^{m_i}_i\}$ of $U_i$ such that $f_1\big|_{U^j_i}\simeq f_2\big|_{U^j_i}\simeq f_n\big|_{U^j_i}$ for all $j\in \{0, 1, \cdots, m_i\}$. 

Notice that the collection $\mathcal{U}=\{U^0_0, U^1_0, \cdots, U^{m_0}_0, U^0_1, U^1_1, \cdots, U^{m_1}_1, \cdots, U^0_k, U^1_k, \cdots, U^{m_k}_k\}$ is an open cover for $X$ such that $f_1\big|_{V}\simeq f_2\big|_{V}\simeq \cdots \simeq f_n\big|_{V}$ for all $V\in \mathcal{U}$. The required inequality follows from the cardinality of $\mathcal{U}$ is $(m_0 + m_1 + \cdots + m_k) +k+1$. 
\end{proof}

The following propositions will be used to give the main results in the third and the fourth sections. 

\begin{proposition}\label{P31} Given maps $f_i: X\rightarrow Y$ and $h_i: Y\rightarrow Z$ for $i\in \{1, 2, \cdots, n\}$. If $h_i \simeq h_{i+1}$ for every $i\in \{1, 2,\cdots, n-1 \}$, then 
\[
\D(h_1 \circ f_1, h_2 \circ f_2, \cdots, h_n \circ f_n) \leq \D(f_1, f_2, \cdots,f_n).  
\]  
\end{proposition}

\begin{proof} Suppose $\D(f_1, f_2, \cdots,f_n)=k$. Then there exists an open covering $\{U_0, U_1, \cdots, U_k\}$ of $X$ such that $f_1|_{U_j} \simeq f_2|_{U_j}\simeq \cdots \simeq f_n|_{U_j}$ for each $j\in \{0, 1, \cdots, k\}$. 

For each $j\in \{0, 1, \cdots, k\}$ and for any distinct $\ell, m \in \{1, 2, \cdots, n\}$, we have 
\[
\Big( h_\ell \circ f_\ell \Big)\Big|_{U_j}\simeq  h_\ell \circ f_\ell\Big|_{U_j} \simeq h_m \circ f_m\Big|_{U_j} \simeq \Big( h_m \circ f_m \Big)\Big|_{U_j}.
\]

Therefore $\D(h_1 \circ f_1, h_2 \circ f_2, \cdots, h_n \circ f_n)\leq k$.

\end{proof}

\begin{proposition}\label{P32} Given maps $f_i: X\rightarrow Y$ and $h_i: Z\rightarrow X$ for $i\in \{1, 2, \cdots, n\}$. If $h_i \simeq h_{i+1}$ for every $i\in \{1, 2,\cdots, n-1 \}$, then 
\[
\D(f_1 \circ h_1, f_2 \circ h_2, \cdots, f_n \circ h_n) \leq \D(f_1, f_2, \cdots,f_n).  
\]  
\end{proposition}

\begin{proof} Suppose $\D(f_1, f_2, \cdots,f_n)=k$. Then there exists an open covering $\{U_0, U_1, \cdots, U_k\}$ of $X$ such that $f_1|_{U_j}\simeq f_2|_{U_j}\simeq \cdots \simeq f_n|_{U_j}$ for each $j\in \{0, 1, \cdots, k\}$. 

Let $V_j:=h^{-1}_{\ell}(U_j)\subseteq Z$. Notice that open subsets $V_j$ cover $Z$. Denote by $h'_{i,j}: V_j\rightarrow U_j$ the restriction on both domain and range, and denote by $\iota_j: U_j\xhookrightarrow{} X$ the inclusion.  

Then for each $j\in \{0, 1, \cdots, k\}$ and for any distinct $\ell, m \in \{1, 2, \cdots, n\}$, we have 
\begin{align*}
\Big(f_\ell \circ h_\ell\Big)\Big|_{U_j} &= {f_\ell}\big|_{U_j} \circ h'_{\ell,j}\simeq {f_m}\big|_{U_j} \circ h'_{\ell,j} \\
&\simeq {f_m}\big|_{U_j} \circ h'_{m,j} = f_m \circ \iota_j \circ h'_{m,j} \\
&= f_m \circ {h_m}\big|_{U_j} = \Big(f_m \circ h_m\Big)\Big|_{U_j}.
\end{align*}
\noindent So $\D(f_1 \circ h_1, f_2 \circ h_2, \cdots, f_n \circ h_n)\leq k$. 
\end{proof}

\begin{lemma}\label{OS} \cite{OS} Let $\mathcal{U}=\{U_0, U_1, \cdots, U_m\}$ and $\mathcal{V}=\{V_0, V_1, \cdots, V_n\}$ be two open coverings of a normal space $X$ such that each set of $\mathcal{U}$ satisfies Property (A) and each set of $\mathcal{V}$ satisfies Property (B). If Property (A) and Property (B) are inherited by open subsets and disjoint unions, then $X$ has an open covering $\mathcal{W}=\{W_0, W_1, \cdots, W_{m+n}\}$ which satisfies both Property (A) and Property (B). 
\end{lemma}

\begin{theorem} Let $X$ be a normal spaceand $n,m \in \mathbb{Z}^+$ with $n\leq m$. If $f_1, \cdots, f_n, g_1, \cdots, g_m, h_1, \cdots, h_m: X \rightarrow Y$ are maps, then
\[
\D(f_1, \cdots, f_n, h_1, \cdots, h_m)\leq \D(f_1, \cdots, f_n, g_{\sigma(1)}, \cdots, g_{\sigma(s)})+\D(g_{\beta(1)}, \cdots, g_{\beta(s')}, h_1, \cdots, h_m)
\]

\noindent where $\sigma$ and $\beta$ is a permutation of $\{1, 2, \cdots, m\}$ and $\{1, 2, \cdots, n\}$ respectively, such that $g_{\sigma(i_0)}\simeq g_{\beta(j_0)}$ for some $i_0\in \{1, 2, \cdots, m\}$ and $j_0 \in \{1, 2, \cdots, n\}$. 
\end{theorem}

\begin{proof} Let $\D(f_1, \cdots, f_n, g_{\sigma(1)}, \cdots, g_{\sigma(s)})=k_1$ and $\D(g_{\beta(1)}, \cdots, g_{\beta(s')}, h_1, \cdots, h_m)=k_2$. So there exists an open covering $\mathcal{U}=\{U_0, \cdots, U_{k_1}\}$ of $X$ such that $f_1 \big|_{U_i}\simeq \cdots \simeq f_n \big|_{U_i}\simeq g_{\sigma(1)}\big|_{U_i}\simeq \cdots g_{\sigma(s)}\big|_{U_i}$ for all $i=0, 1, \cdots, k_1$. Similarly there exists an open covering $\mathcal{V}=\{V_0, \cdots, V_{k_2}\}$ such that $g_{\beta(1)}\big|_{V_j}\simeq \cdots g_{\sigma(s')}\big|_{V_j} h_1 \big|_{V_j} \simeq h_m \big|_{V_j}$ for all $j=0, 1, \cdots, k_2$.

From the assumption, $g_{\sigma(i_0)}\simeq g_{\beta(j_0)}$ for some $i_0$ and $j_0$, and by Lemma~\ref{OS}, we have an open covering $\mathcal{W}=\{W_0, \cdots, W_{k_1 + k_2}\}$ of $X$ such that $f_1 \big|_{W_k}\simeq \cdots \simeq f_n \big|_{W_k} \simeq g_1 \big|_{W_k}\simeq \cdots \simeq g_m \big|_{W_k} \simeq h_1 \big|_{W_k}\simeq \cdots \simeq h_m \big|_{W_k}$ for all $k=0, 1, \cdots, k_1 + k_2$. 
\end{proof}

\section{$\TC_n$, $\cat$, $\secat$ of a fibration and $n$-th homotopic distance}

In the first half of this section we will introduce the relation between $n$-th homotopic distance with $\cat$ and with $\TC_n$ and with $\secat$ of a fibration. In the second half, we will prove some properties of $\TC_n$ such as its relation with $\cat$, via higher homotopic distance. 

Let us begin this section by recalling the definitions of $\cat$, $\secat$ and $\TC_n$. 

\begin{definition} \cite{CLOT} Lusternik Schnirelmann category $\cat(X)$ of a space $X$ is the least integer $k\geq 0$ such that there exists an open covering $\{U_0, U_1, \cdots, U_k\}$ of $X$ with the property that the inclusion on each $U_i$ is null-homotopic. Such an open $U_i$ is usually called  \textit{categorical}.
\end{definition}

\begin{definition} \cite{S} Sectional category $\secat(q)$ of a fibration $q: E \rightarrow B$ is the least integer $k\geq 0$ such that there exists an open covering $\{U_0, U_1, \cdots, U_k\}$ of $B$ such that there is a section $s_i$ over $U_i$ for every $i = 0, 1, \cdots, k$.
\end{definition}

A special case of the sectional category, $\TC_n$, is defined as follows.

\begin{definition}\label{D}\cite{R} For $n \in \mathbb{N}$, let $J_{n}$ be the wedge sum of $n$ closed intervals $[0,1]_{i}$ for $i=1,2,\cdots,n$ where the zeros $0_{i}$'s are identified. For a path-connected space $X$, denote by $X^{J_{n}}$ the space of paths with $n$-legs. Then there is a fibration $e_{n}:X^{J_{n}}\to X^{n}$ defined by $e_{n}(f)=(f(1_{1}),f(1_{2}), \cdots, f(1_{n}))$, and the sectional category (or Schwarz genus) of this fibration is called $n$-dimensional topological complexity of $X$, denoted by $\TC_{n}(X)$. 

If there is no such a covering, we write $\TC_{n}(X)=\infty$.
\end{definition}

\begin{theorem}\label{TV} For a fixed $x_0 \in X$, consider the inclusions $j_i: X^{n-1}\hookrightarrow X^n$ given by  $j_i(x_1, \cdots, x_{n-1})=(x_1, \cdots, x_{i-1}, x_0, x_i, x_{i+1}, \cdots, x_{n-1})$ for $i\in \{1,2, \cdots, n\}$. Then $D(j_1, \cdots, j_n)=\cat(X^{n-1})$.
\end{theorem}

\begin{proof} Let $U\subset X^{n-1}$ be categorical. 

Our aim is to show that each $j_i$ restricted on $U$ is homotopic with each other for $i\in \{1,2, \cdots, n\}$. It suffices to show that for any distinct $i, k \in \{1, \cdots, n\}$, we have $j_i|_{U}\simeq j_k|_{U}$.

For convenience, we will write $x=(x_1, \cdots, x_n)$. Consider the homotopy 
\[
F:U\times I\rightarrow X^{n-1}, \text{ satisfying } F(x,0)=x \text{ and } F(x,1)=(x_0, \cdots, x_0). 
\]

Define the map $H:U\times I\rightarrow X^{n}$ by 
\[
H(x,t)=\begin{cases} 
      j_i(F(x,2t)), & \hspace{0.1in} 0\leq t\leq \frac{1}{2}	 \\
      j_k(F(x,2-2t)), & \hspace{0.1in}  \frac{1}{2}\leq t\leq 1	
   \end{cases}
\]

\noindent so that $H(x,0)=j_i(x)$, $H(x,1)=j_k(x)$ are satisfied. Since 
\[
H\Big(x,\frac{1}{2}\Big)=j_i(F(x,1))=j_i(x_0,\cdots, x_0)=(x_0, \cdots, x_0)=j_k(F(x,1)),
\]

\noindent $H$ is continuous. Hence $\D(j_1, \cdots, j_n)\leq \cat(X^{n-1})$. 

For the other way around, assume that we have ${j_1}|_U \simeq \cdots \simeq {j_n}|_U$, i.e., there exists a homotopy $H_i: U\times I\rightarrow X^n$ such that $H_i(x,0)=j_i(x)$ and $H_i(x,1)=j_{i+1}(x)$ for each $i \in \{1, \cdots, n-1\}$.

Define $F:U\times I\rightarrow X^{n-1}$ by
\[
F(x_1, \cdots, x_{n-1},t)=\big( \pr_2(H_1(x,t)), \cdots, \pr_{i}(H_{i-1}(x,t)), \cdots, \pr_n(H_{n-1}(x,t)) \big) 
\]

\noindent where $\pr_i:X^n \rightarrow X$ is the projection maps into the $i$-th factor and which satisfies 
\begin{align*}
F(x,0) &= \big( \pr_2(H_1(x,0)), \cdots, \pr_{i}(H_{i-1}(x,0)), \cdots, \pr_n(H_{n-1}(x,0)) \big) \\
 &= \big( \pr_2(j_1(x)), \cdots, \pr_{i}(j_{i-1}(x)), \cdots, \pr_n(j_{n-1}(x)) \big)  \\
 &= (x_1, \cdots, x_n) 
\end{align*}
and 
\begin{align*}
F(x,1) &= \big( \pr_2(j_2(x)), \cdots, \pr_{i}(j_{i}(x)), \cdots, \pr_n(j_{n}(x)) \\
&= (x_0, \cdots, x_0). 
\end{align*}
Thus $\cat(X^{n-1})\leq \D(j_1, \cdots, j_n)$.
\end{proof}

\begin{theorem}\label{T28} Let $f_1, f_2, \cdots, f_n: X\rightarrow Y$ be maps. Consider the fibration $e_n:Y^{J_n}\rightarrow Y^n$ as defined in Definition~\ref{D}. If $q: P\rightarrow X$ is the pullback of the fibration $e_n$ by the map $F:=(f_1, f_2, \cdots, f_n): X \rightarrow Y^n$, then $\D(f_1, f_2, \cdots, f_n)=\secat(q)$.

\[
\begin{tikzcd}
P \arrow[r] \arrow[d, "q"]
& Y^{J_n} \arrow[d, "e_n"] \\
X \arrow[r,  "F" ]
& Y^n
\end{tikzcd}
\]
\end{theorem}

\begin{proof} Before we start proving the theorem, observe that 
\begin{align*}
P&= \{(x,\gamma)\in X\times Y^{J_n} \ \big| \  e_n(\gamma)= F(x)\} \\
&= \{(x,\gamma)\in X\times Y^{J_n} \  \big| \  \gamma_i(1_i)=f_i(x) \text{ for all } i=1, 2, \cdots, n\}.
\end{align*}

Suppose $\D(f_1, f_2, \cdots, f_n)=k$. Then there exists an open covering $\{U_0, U_1, \cdots, U_k\}$ of $X$ such that $f_1\big|_{U_j}\simeq f_2\big|_{U_j}\simeq \cdots \simeq f_n\big|_{U_j}$ for all $j\in \{0, 1, \cdots, k\}$. 

For each $U_j$, we have the homotopies 
\[ \widetilde{H}_1: f_1\big|_{U_j} \simeq f_2\big|_{U_j}
\]
\[
\cdots
\]
\[ \widetilde{H}_{n-1}: f_{n-1}\big|_{U_j} \simeq f_n\big|_{U_j}. 
\]

Fix one of the homotopies, say $\widetilde{H}_1$. Write the function $\widetilde{H}_1(x,\frac{1}{2})=:g(x)$. So there are new homotopies 
\[ {H}_1: g\big|_{U_j} \simeq f_1\big|_{U_j}
\]
\[
\cdots
\]
\[ {H}_{n}: g\big|_{U_j} \simeq f_n\big|_{U_j}. 
\]

Define a continuous map  

\begin{align*}
s_j :  U_j & \longrightarrow Y^n  \longrightarrow  Y^{J_n} \\ 
 x & \mapsto  F(x)  \mapsto  \beta_x
\end{align*}

\noindent where $\beta_x: J_n \rightarrow Y$ defined by $\beta_x\big|_{I_i}=H_i$ with $\beta_x(0_1)=\cdots = \beta_x(0_n)=g(x)$ are glued.

Since 
\[
\beta_x(1_i)=H_i(x,1_i)= f_i(x) \text{ for each } i = 1, 2, \cdots, n, 
\]

\noindent each $s_j$ is  a section of $q$ over $U_j$. Hence $\secat(q)\leq k$.
 
On the other way around, if $\secat(q)=k$ then we have an open covering $\{U_0, \cdots, U_k\}$ of $X$ such that there is a section $s_j:U_j\rightarrow P$ (i.e., $q\circ s_j=\id_{U_j}$) for each $j=0,1, \cdots, k$. 

Since $q$ maps to the first factor, $s_j: U_j \rightarrow P$ is defined by $s(x)=(x,\gamma_x)$ such that $\gamma_x$ satisfies 
\[
\gamma_x(1_1)= f_1(x), \gamma_x(1_2)= f_2(x), \cdots, \gamma_x(1_n)= f_n(x) 
\]
and 

\[
\gamma_x(0_1)= \gamma_x(0_2)= \cdots =  \gamma_x(0_n). 
\]

Let say $\gamma_x(0_i)=y_0=g(x)$ for some map $g: X\rightarrow Y$, for some (and all) $i$. 

Each $i$-th leg gives a path, denote it by $H^i_x:I_i\rightarrow Y$ satisfying $H^i_x(0)=g(x)$ and $H^j_x(1_i)=f_i(x)$. So over each $U_j$, all $f_i$'s are homotopic to each other. Thus $\D(f_1, \cdots, f_n)\leq k$. 

\end{proof}

\begin{corollary}\label{C36} If $f_1, f_2, \cdots, f_n: X\rightarrow Y$ are maps with path connected spaces $X$ and $Y$, then 
\[
\D(f_1, f_2, \cdots, f_n)\leq \cat(X).
\] 

\end{corollary}

\begin{proof} Suppose $\cat(X)=k$. Then there exists an open covering $\{U_0, U_1, \cdots, U_k\}$ of $X$ such that $\id\big|_{U_i}\simeq c\big|_{U_i}$ for all $i\in\{0, 1, \cdots, k\}$ where $c$ is a constant map. Let say $H_i$ is the homotopy between these maps for each $i$ with $H_i(x,0)=\id\big|_{U_i}(x)$.

\[
\begin{tikzcd}
& P \arrow[r] \arrow[d, "q"] & Y^{J_n} \arrow[d, "e_n"] \\
U_i \times I \arrow[r, "H_i"]   \arrow{ur}{\widetilde{H}_i}  & X \arrow[r, "F"]  & Y^n
\end{tikzcd}
\]

By the homotopy lifting property, there exists $\widetilde{H}_i$ such that $q\circ \widetilde{H}_i=H_i$ and $H_i(x,0)=\id\big|_{U_i}$ is lifted to some $\widetilde{f}_i$. If we choose the sections as $\widetilde{f}_i$ for each $i$, then we have $q\circ \widetilde{f}_i=\id\big|_{U_i}$ as required. Hence this shows that $\secat(q) \leq k$. Summarizing, we have $\D(f_1, f_2, \cdots, f_n)=\secat(q) \leq \cat(X)$ where the equality follows from Theorem~\ref{T28}.
\end{proof}

\begin{remark} Notice that the necessary condition in Theorem~\ref{TV} also follows from Corollary~\ref{C36}, if we take $f_i=j_i: X^{n-1} \rightarrow X^n$. 
\end{remark}

\begin{corollary} If $f_1, f_2, \cdots, f_n: X\rightarrow Y$ are maps, then $\D(f_1, f_2, \cdots, f_n)\leq \TC_n(Y)$. 
\end{corollary}

\begin{proof} If the pullback of $e_n:Y^{J_n}\rightarrow Y^n$ is the fibration $q:P\rightarrow X$, then as mentioned in \cite{MVML}, $\secat(q)\leq \secat(e_n)$. Hence by Theorem~\ref{T28}, $\D(f_1, f_2, \cdots, f_n)=\secat(q)\leq \TC_n(Y)$.
\end{proof}

\begin{lemma}\label{L}\cite{B}  Let $X$ be a path-connected space and consider the fibration $e_n$ as described in Definition~\ref{D}. Let $U\subseteq X$ be an open. Then there is a section $s:U\to X^{J_n}$ of $e_n$ if and only if each composition $f_i:U \xhookrightarrow{\iota} X^n  \xrightarrow{\pr_i} X $ is homotopic to some map $g\colon U \to X$, where $\pr_i:X^n \rightarrow X$ is the projection to the $i$-th factor. 
\end{lemma}

The following theorem whose proof follows directly from Lemma~\ref{L} is introduced in \cite{MVML}. 

\begin{theorem}\cite{MVML}\label{T2} If $X$ is path-connected space and each map $\pr_i: X^n\rightarrow X$ projects to the $i$-th factor for $i\in \{1,2, \cdots, n\}$, then $\D(\pr_1,\pr_2, \cdots, \pr_n)=\TC_n(X)$.  \\
\end{theorem}

Although Corollary~\ref{R123} and Theorem~\ref{BGRT123} are already known, here we will give new alternative proofs using higher homotopic distance. 

\begin{corollary}\cite{R}\label{R123} $\TC_n(X)\leq \TC_{n+1}(X)$.
\end{corollary}

\begin{proof} For convenience, let us denote the projection maps by two different notations depending on their domains, that is, $\pr_i: X^n\rightarrow X$ and $\overline{pr}_i: X^{n+1}\rightarrow X$ which project onto the $i$-th factor. 

Suppose $\TC_{n+1}(X)=k$. By Theorem~\ref{T2}, there exists an open covering $\{U_0, U_1, \cdots, U_k\}$ of $X^{n+1}$ such that $\overline{\pr}_1\big|_{U_j}\simeq \overline{\pr}_2\big|_{U_j}\simeq \cdots \simeq \overline{\pr}_{n+1}\big|_{U_j}$ for all $j \in \{0, 1, \cdots, k\}$. 

Notice that $\overline{\pr}_i\big|_{X^n}=\pr_i$ for any $i\in \{1, 2, \cdots, n\}$. So 
\[
\pr_m\big|_{U_j} = \overline{\pr}_m\big|_{U_j\cap X^n} \simeq \overline{\pr}_\ell\big|_{U_j\cap X^n} = \pr_\ell\big|_{U_j} 
\]
for any distinct $m, \ell \in \{1, 2, \cdots, n\}$ and for all $j \in \{0, 1, \cdots, k\}$. Hence $\D(\pr_1, \pr_2, \cdots, \pr_n)=\TC_{n}(X)\leq k$.

\end{proof}

\begin{theorem}\cite{BGRT}\label{BGRT123} $\TC_n(X)\leq \cat(X^n)$.
\end{theorem}

\begin{proof}  For a fixed $x_0 \in X$, consider the inclusions
\[
j_\ell: X^n \hookrightarrow X^{n+1} \hspace{0.02in}\text{ given by }\hspace{0.02in} j_\ell(x_1, \cdots, x_n)= (x_1, \cdots, x_{i-1}, x_0, x_{i+1}, \cdots, x_n).
\]

\noindent for $\ell = 1, 2, \cdots, n$.

Let suppose $\cat(X^n)=k$. Then there exists an open covering $\{U_0, U_1, \cdots, U_k\}$ of $X^n$ such that $j_\ell\big|_{U_i}\simeq j_{\ell+1}\big|_{U_i}$ for all $i$ and for all $\ell$. Let $x=(x_1, x_2, \cdots, x_n)$ and let say $F_i^\ell: U_i\times I \rightarrow X^{n+1}$ be the homotopy such that $F_i^\ell(x,0)=j_\ell(x)$ and $F_i^\ell(x,1)=j_{\ell+1}(x)$ for every $\ell$.

On each $U_i$, define a homotopy $H_i^\ell:U_i\times I\rightarrow X$ by $H_i^\ell(x,t)=(\pr_\ell\circ F_i^\ell)(x,t)$ where $\pr_\ell: X^n\rightarrow X$ is the projection that maps onto the $\ell$-th factor. Since 
\begin{align*}
H_i^\ell(x,0)=\pr_\ell(F_i^\ell(x,0))=\pr_\ell(j_\ell(x))=(x_1, x_2, x_{\ell-1}, x_0, x_\ell, \cdots, x_n)=x_0
\end{align*}
and
\begin{align*}
H_i^\ell(x,1)&=\pr_\ell(F_i^\ell(x,1))=\pr_\ell(j_{\ell+1}(x)) \\
&=(x_1, x_2, x_{\ell}, x_0, x_{\ell+1}, \cdots, x_n) =x_\ell \\  
&=\pr_\ell(x_1, \cdots, x_n)
\end{align*}

\noindent for each $\ell = 1, 2, \cdots, n$, we have $\pr_1\big|_{U_i} \simeq \pr_2\big|_{U_i} \simeq \cdots \simeq \pr_n\big|_{U_i}$ for all $i=0, 1, \cdots, k$. Thus $\TC_n(X)=\D(\pr_1, \pr_2, \cdots, \pr_n)\leq k$ by Lemma~\ref{T2}. 

\end{proof}

\begin{theorem} Let $f_1, f_2, \cdots, f_n: X\rightarrow Y$ and $g_1, g_2, \cdots, g_n: \widetilde{X}\rightarrow \widetilde{Y}$ be maps. If $X\times \widetilde{X}$ is normal space, then we have
\[
\D(f_1\times g_1, f_2\times g_2, \cdots, f_n\times g_n)\leq \D(f_1, f_2, \cdots, f_n)+\D(g_1, g_2, \cdots, g_n).
\]
\end{theorem}

\begin{proof} Let $\D(f_1, f_2, \cdots, f_n)=k_1$ and $\D(g_1, g_2, \cdots, g_n)=k_2$. So there exists an open covering $\{U_0, \cdots, U_{k_1}\}$ of $X$ such that $f_1\big|_{U_i}\simeq f_2\big|_{U_i}\simeq \cdots\simeq f_n\big|_{U_i}$ for all $i$ and there exists an open covering $\{V_0, \cdots, V_{k_2}\}$ of $\widetilde{X}$ such that $g_1\big|_{V_j}\simeq g_n\big|_{V_j}\simeq \cdots\simeq g_n\big|_{V_j}$ for all $j$. Then we have homotopies $\widetilde{F_s}:f_s\simeq f_{s+1}$ for each $s=1,2,\cdots, {n-1}$. By the argument in Theorem~\ref{T28}, we can find new homotopies between $f_s$ and a fixed map $f$ for all $s=1,2, \cdots, n$. Similarly, we can find new homotopies between $g_s$ and a fixed map $g$.

Let say that Property (A) is the following: On each $U_i$, there exists homotopies $f_s \times \id_{\widetilde{X}}\simeq f \times \id_{\widetilde{X}}$ for all $s=1, 2, \cdots, n$. Similarly, let Property (B) be the following: On each $V_i$, there exists homotopies $\id_{X}\times g_s\simeq \id_{X} \times g$ for all $s$. 

$\mathcal{U}=\{U_1\times \widetilde{X}, \cdots, U_{k_1}\times \widetilde{X}\}$ is a covering of $X\times \widetilde{X}$ satisfying Property (A) and $\mathcal{V}=\{X\times V_1, \cdots, X\times V_{k_2}\}$ is a covering of $X\times \widetilde{X}$ satisfying Property (B). So by Lemma~\ref{OS}, there exists $\mathcal{W}=\{W_0, W_1, \cdots, W_{k_1 + k_2}\}$ which satisfies both Property (A) and Property (B). Therefore, we have the following homotopies

\[ F_m: f_m \times \id\big|_{\widetilde{X}} \simeq f \times \id\big|_{\widetilde{X}}
\]
and 
\[ G_m: \id\big|_{X} \times g_m \simeq \id\big|_{X} \times g
\]

\noindent where $F_m:W_s\times I\rightarrow Y\times \widetilde{Y}$ and $G_m : W_s\times I\rightarrow Y\times \widetilde{Y}$ for all $m=1,2,\cdots, n$ and for $s=1,2,\cdots, k_1 +k_2$.

For $m=1,2,\cdots, n$, define $H_m :=\big( \pr_1 \circ F_m, \pr_{2}\circ G_m \big)$ where $\pr_1:Y\times \widetilde{Y}\rightarrow Y$ and $\pr_2:Y\times \widetilde{Y}\rightarrow \widetilde{Y}$ projection maps onto the first and the second factor, respectively. Hence for $z\in X\times \widetilde{X}$,

\begin{align*}
H_m(z,0) &= \big( \pr_1 \circ F_m(z,0), \pr_{2}\circ G_m(z,0)\big) \\
&=\big( \pr_1 (f_m\times\id\big|_{\widetilde{X}}(z)), \pr_{2}(\id\big|_{X} \times g_m(z))\big)\\
&= \big( f_m (\pr_1(z)), \pr_{2}(g_m (\pr_2(z))\big)\\
&= (f_m\times g_m)(z).
\end{align*}

and 

\begin{align*}
H_m(z,0) &= \big( \pr_1 \circ F_m(z,1), \pr_{2}\circ G_m(z,1)\big) \\
&=\big(\pr_1 (f\times\id\big|_{\widetilde{X}}(z)), \pr_{2}(\id\big|_{X} \times g(z))\big)\\
&= \big(f (\pr_1(z)), \pr_{2}(g (\pr_2(z))\big)\\
&= (f\times g)(z).
\end{align*}

So each $f_m\times g_m$ is homotopic to each other for all $m=1,2, \cdots, n$ on each $W_s$. Thus $\D(f_1\times g_1, f_2\times g_2, \cdots, f_n\times g_n)\leq k_1 + k_2$.
\end{proof}

\begin{corollary} \cite{BGRT} $\TC_n (X_1 \times X_2) \leq \TC_n(X_1)+\TC_n(X_2)$.
\end{corollary}

\begin{proof} Take the maps $\pr_1^i, \pr_2^i, \cdots, \pr_n^i: X_i^n\rightarrow X_i$ for $i=1,2$ where $\pr_j^i$ denotes the projection onto the $j$-th factor. Then 
\[
\TC_n(X_1 \times X_2) = \D(\pr_1^1\times \pr_1^2, \cdots, \pr_n^1 \times \pr_n^2) \leq  \Sigma_{i=1}^{2} \D(\pr_1^i, \cdots, \pr_n^i,)=  \TC_n(X_1)+\TC_n(X_2).
\]

\end{proof}

\section{Homotopy Invariance of the Higher Homotopic Distance}

\begin{theorem} Higher homotopic distance is homotopy invariance in the sense that if $X\simeq X'$ and $Y\simeq Y'$ homotopy equivalent spaces connecting the maps $f_i:X\rightarrow Y$ and $g_i: X'\rightarrow Y'$ for all $i \in \{1, 2, \cdots, n\}$, then $\D(f_1, f_2, \cdots, f_n)=\D(g_1, g_2, \cdots, g_n)$.
\end{theorem}

\begin{proof} If $X\simeq X'$, then there exist maps $\alpha: X\rightarrow X'$ and $\beta: X'\rightarrow X$ such that $\alpha \circ \beta \simeq \id_{X'}$ and $\beta \circ \alpha \simeq \id_{X}$.

By Proposition~\ref{P32}, $\D( f_1\circ \beta, f_2 \circ \beta, \cdots, f_n\circ \beta) \leq \D(f_1, f_2, \cdots,f_n)$. On the other hand, 
\begin{align*}
\D(f_1, f_2, \cdots,f_n) & = \D( f_1\circ \id_X, f_2 \circ \id_X, \cdots, f_n\circ \id_X) \\
&=\D( f_1\circ \beta \circ \alpha, f_2 \circ \beta \circ \alpha, \cdots, f_n\circ \beta \circ \alpha) \\
&\leq \D( f_1\circ \beta , f_2 \circ \beta , \cdots, f_n\circ \beta) 
\end{align*}

\noindent where the inequality follows from Proposition~\ref{P32} and the second equality follows from Proposition~\ref{A}. So $\D( f_1\circ \beta, f_2 \circ \beta, \cdots, f_n\circ \beta) = \D(f_1, f_2, \cdots,f_n)$.

If $Y\simeq Y'$, then there exist maps $\gamma: Y\rightarrow Y'$ and $\eta: Y'\rightarrow Y$ such that $\gamma \circ \eta \simeq \id_{Y'}$ and $\eta \circ \gamma \simeq \id_{Y}$. In a similar way, using Proposition~\ref{P31}, one can see that $\D(g_1, g_2, \cdots, g_n)=\D(\beta_Y \circ g_1, \eta \circ g_2, \cdots, \eta \circ g_n)$. Hence 
\begin{align*}
\D(f_1, f_2, \cdots,f_n) &= \D(f_1 \circ \beta, f_2\circ \beta, \cdots, f_n\circ \beta) \\
&=\D(\eta \circ g_1, \eta \circ g_2, \cdots, \eta \circ g_n) \\
&= \D(g_1, g_2, \cdots, g_n).
\end{align*}
\end{proof}

\begin{corollary}\cite{R} $\TC_n$ is homotopy invariant.
\end{corollary}

\end{document}